\documentclass[11pt]{amsart}
\usepackage{amssymb}
\usepackage{amsthm}

\usepackage{graphicx}

\def\inv{{^{-1}}}

\def\Z{{\mathbb Z}}

\def\g{{\gamma}}

\def\TSG{{\mathrm{TSG_+}}}
\def\Aut{{\mathrm{Aut}}}
\def\Diff{{\mathrm{Diff_+}}}
\def\fix{{\mathrm{fix}}}

\def\INT{{\mathrm{int}}}

\newtheorem*{A4}{$A_4$ Theorem}
\newtheorem*{S4}{$S_4$ Theorem}
\newtheorem*{A5}{$A_5$ Theorem}
\newtheorem*{isometry}{Isometry Theorem}
\newtheorem*{orbits}{Orbits Lemma}
\newtheorem*{D2}{$D_2$ Lemma}

\newtheorem*{complete}{Complete Graph Theorem}

\def\fix{{\mathrm{fix}}}

\newcommand{\gen}[1]{\langle #1 \rangle}

\newtheorem{prop}{Proposition}

\newtheorem{thm}{Theorem}

\newtheorem*{Graph}{Edge Embedding Lemma}

\newtheorem*{Subgroup}{Subgroup Theorem}
\newtheorem*{fixed vertex}{Fixed Vertex Property}

\newtheorem{lemma}{Lemma}

\title[Polyhedral topological symmetry groups]{Complete graphs whose topological symmetry groups are polyhedral}
\author{Erica Flapan, Blake Mellor, and Ramin Naimi}


\thanks{The authors were supported in part by NSF Grants DMS-0905087, DMS-0905687, DMS-0905300.}


\address{Department of Mathematics, Pomona College, Claremont, CA 91711, USA}

\address{Department of Mathematics, Loyola Marymount University, Los Angeles, CA 90045, USA}

\address{Department of Mathematics, Occidental College, Los Angeles, CA 90041, USA}

\begin{document}
\date \today
\maketitle
\begin{abstract}
We determine for which $m$, the complete graph $K_m$ has an embedding in $S^3$ whose topological symmetry group is isomorphic to one of the polyhedral groups: $A_4$, $A_5$, or $S_4$.
\end{abstract}

\section{Introduction}
Characterizing the symmetries of a molecule is an important step in predicting its chemical behaviour.   Chemists have long used the group of rigid symmetries, known as the {\it point group}, as a means of representing the symmetries of a molecule.  However, molecules which are flexible or partially flexible may have
symmetries which are not included in the point group. Jon Simon  \cite {Si} introduced the concept of the
{\it topological symmetry group} in order to study symmetries of such non-rigid molecules. 
The topological symmetry group provides a way to classify, not only the symmetries of molecular graphs, but the symmetries of any
graph embedded in $S^3$. 

We define the topological symmetry group as follows.  Let $\gamma$ be an abstract graph, and let $\Aut(\gamma)$ denote the automorphism group of $\gamma$.  Let $\Gamma$ be the image of an embedding of $\gamma$ in $S^3$.  The \emph{topological symmetry group} of $\Gamma$, denoted by $\mathrm{TSG}(\Gamma)$, is the subgroup of $\Aut(\gamma)$ which is induced by homeomorphisms of the pair $(S^3,\Gamma)$. The {\it orientation preserving topological symmetry group} of $\Gamma$, denoted by $\TSG(\Gamma)$, is the subgroup of $\Aut(\gamma)$ which is induced by orientation preserving homeomorphisms of the pair $(S^3,\Gamma)$. In this paper we are only concerned with $\TSG(\Gamma)$, and thus for simplicity we abuse notation and refer to the group $\TSG(\Gamma)$ simply as the {\it topological symmetry group} of $\Gamma$.

Frucht \cite{Fr} showed that every finite group is 
the automorphism group 
of some connected graph.  Since every graph can be embedded in $S^3$,
it is natural to ask whether every finite group 
can be realized as  
$\TSG(\Gamma)$ 
for some connected graph $\Gamma$ embedded in $S^3$.   Flapan, Naimi, Pommersheim, and Tamvakis proved in \cite{FNPT} that the answer to this question is 
``no'', and proved that there are strong restrictions on which groups can occur as topological symmetry groups.  For example, it was shown that $\TSG(\Gamma)$ can never be the alternating group $A_n$ for $n>5$.  

The special case of topological symmetry groups of complete graphs  is interesting to consider
because a complete graph $K_n$ has the largest automorphism group of any graph with $n$ vertices. In \cite{FNT}, Flapan, Naimi, and Tamvakis characterized which finite groups can occur as topological symmetry groups of embeddings of complete graphs in $S^3$ as follows.  

\medskip \begin{complete} \cite{FNT}
\label{T:TSG2} 
A finite group $H$ is isomorphic to $\TSG(\Gamma)$ for some embedding $\Gamma$ of a complete graph in $S^3$ if and only if $H$ is a finite cyclic group, a dihedral group, a subgroup of $D_m \times D_m$ for some odd $m$, or $A_4$, $S_4$, or $A_5$.
\end{complete}

We use $D_m$ to denote the dihedral group with $2m$ elements.  The groups $A_4$, $S_4$, or $A_5$, are known as {\it polyhedral groups} because they consist of: the group of rotations of a tetrahedron (which is isomorphic to $A_4$), the group of rotations of a cube or octahedron (which is isomorphic to $S_4$), and  the group of rotations of a dodecahedron or icosahedron (which is isomorphic to $A_5$).

Observe that the Complete Graph Theorem does not tell us which complete graphs can have a given group $H$ as their topological symmetry group.  In this paper we characterize which complete graphs can have each of the polyhedral groups as its topological symmetry group.  In particular, in the following results we determine for which $m$, $K_m$ has an embedding $\Gamma$  with $\TSG(\Gamma)\cong A_4$, $A_5$, or $S_4$.  
\begin{A4}
A complete graph $K_m$ with $m\geq 4$ has an embedding $\Gamma$ in $S^3$ such that  $\TSG(\Gamma) \cong A_4$ if and only if $m \equiv 0$, $1$, $4$, $5$, $8 \pmod {12}$.
\end{A4}

\begin{A5}
A complete graph $K_m$ with $m\geq 4$ has an embedding $\Gamma$ in $S^3$ such that $\TSG(\Gamma) \cong A_5$ if and only if $m \equiv 0$, $1$, $5$, $20 \pmod{60}$.
\end{A5}

\begin{S4}
A complete graph $K_m$ with $m\geq 4$ has an embedding $\Gamma$ in $S^3$ such that  $\TSG(\Gamma) \cong S_4$ if and only if $m \equiv 0$, $4$, $8$, $12$, $20 \pmod {24}$.
\end{S4}

Observe that if $K_n$ has an embedding with topological symmetry group isomorphic to $A_5$ or $S_4$, then $K_n$ also has an embedding with topological symmetry group isomorphic to $A_4$.

In \cite{FMNY} we characterize which complete graphs can have a cyclic group, a dihedral group, or another subgroup of $D_m \times D_m$ as its topological symmetry group.  

\bigskip


\section{Necessity of the conditions} \label{S:nonexistence}

In this section, we prove the necessity of the conditions given in the $A_4$, $A_5$ and $S_4$ Theorems.   We begin by listing some results that were proved elsewhere that will be useful to us.

\begin{orbits}\cite{CFO} \label{orbit lemma}
If $\alpha$ and $\beta$ are permutations of a finite set such that $\alpha$ and $\beta$ commute, then $\beta$ takes $\alpha$-orbits to $\alpha$-orbits of the same length.\end{orbits}

\medskip

\begin{D2}\cite{CFO} \label{P:4n+3} If $m\equiv 3\pmod{4}$, then there is no embedding $\Gamma$ of $K_{m}$ in $S^3$  such that $\TSG(\Gamma)$ contains a subgroup isomorphic to $D_2$.\end{D2} 

\medskip

Note that  each of the groups $A_4$, $A_5$, and $S_4$ contains $D_2$ as a subgroup.  Recall that the groups $A_4$ and $A_5$ can be realized as the group of rotations of a solid tetrahedron and a solid dodecahedron respectively.  Looking at each of these groups of rotations we see that any two cyclic subgroups of the same order are conjugate.  The group $S_4$ can be realized as the group of rotations of a cube.  It follows that all cyclic groups of order 3 or order 4 are conjugate.  Up to conjugacy, $S_4$ contains two cyclic groups of order 2, those which are contained in $A_4$ and those which are not.  This implies the following observation that we will make use of in this section. 
  
\begin{fixed vertex}
Let $G\cong A_4$, $A_5$ and suppose $G$ acts faithfully on a graph $\Gamma$.  Then all elements of $G$ of a given order fix the same number of vertices.  Furthermore, since all of the non-trivial elements of $G$ have prime order, all of the elements in a given cyclic subgroup fix the same vertices.  

Let $H$ be isomorphic to $S_4$ and suppose that $H$ acts faithfully on $\Gamma$.  Then all elements of $H$ of order 3 fix the same number of vertices, and all elements of $H$ of order 4 fix the same number of vertices.  All involutions of $H$ which are in $G\cong A_4$ fix the same number of vertices, and all involutions of $H$ which are not in $G$ fix the same number of vertices.

\end{fixed vertex}

\medskip

We will also use the theorem below to focus on embeddings $\Gamma$ of $K_m$ in $S^3$ such that $\TSG(\Gamma)$ is induced by an isomorphic finite subgroup of $\mathrm{SO}(4)$ (the group of orientation preserving isometries of $S^3$).    This theorem follows from a result in \cite{FNPT} together with the recently proved Geometrization Theorem \cite{MF}.
\medskip 

\begin{isometry} 
\label{SO(4)}
Let $\Omega$ be an embedding of some $K_{m}$ in $S^3$.  Then $K_m$ can be re-embedded in $S^3$ as $\Gamma$ such that $\TSG(\Omega) \leq
\TSG(\Gamma)$ and $\TSG(\Gamma)$ is induced by an isomorphic finite
subgroup of $\mathrm{SO}(4)$.
\end{isometry}
\medskip

Suppose that $\Omega$ is an embedding of a complete graph $K_m$ in $S^3$ such that $G=\TSG(\Omega)$ is isomorphic to $A_4$, $A_5$ or $S_4$.  By applying the Isometry Theorem, we obtain a re-embedding $\Gamma$ of $K_m$ in $S^3$ such that $G\leq \TSG(\Gamma)$ is induced on $\Gamma$ by an isomorphic finite subgroup $\widehat{G}\leq \mathrm{SO}(4)$.   This simplifies our analysis since every finite order element of $\mathrm{SO}(4)$ is either a rotation with fixed point set a geodesic circle or a glide rotation with no fixed points.   If the fixed point sets of two such rotations intersect but do not coincide, then they intersect in 2 points. Furthermore, if all of the elements of a finite subgroup of $\mathrm{SO}(4)$ pointwise fix the same simple closed curve, then that subgroup must be cyclic (this can be seen by looking at the action of the subgroup on the normal bundle).

For each $g\in G$, we let $\widehat{g}$ denote the element of $\widehat{G}$ which induces $g$. Since  $G$ has finite order, if $g\in G$ fixes both vertices of an edge, then $\widehat{g}$ pointwise fixes that edge. Since the fixed point set of every element of $\widehat{G}$ is either a circle or the empty set, no non-trivial element of $G$ can fix more than 3 vertices of $\Gamma$.  If $g\in G$ fixes 3 vertices, then $\fix(\widehat{g})$ is precisely these 3 fixed vertices together with the edges between them.  Suppose that $g\in G$ fixes 3 vertices and has order 2.  Then $g$ must interchange some pair of vertices $v$ and $w$ in $\Gamma$.   Thus $\widehat{g}$ must fix a point on the edge $\overline{vw}$.  As this is not possible, no order 2 element of $G$ fixes more than 2 vertices.  Since $G\leq \Aut(K_m)$ and $G$ is isomorphic to $A_4$, $A_5$ or $S_4$, $m\geq 4$.  In particular, since no $g\in G$ fixes more than 3 vertices, each $g\in G$ is induced by precisely one $\widehat{g}\in \widehat{G}$.  The following lemmas put further restrictions on the number of fixed vertices of each element of a given order. 
  
\medskip

\begin{lemma}\label{n2not2}  Let $G\leq \Aut(K_m)$ which is isomorphic to $A_4$ or $A_5$.  Suppose there is an embedding $\Gamma$ of $K_{m}$ in $S^3$ such that $G$ is induced on $\Gamma$ by an isomorphic subgroup $\widehat{G}\leq\mathrm{SO}(4)$.  Then no order 2 element of $G$ fixes more than 1 vertex of $\Gamma$.
\end{lemma}

\begin{proof}  As observed above, no order 2 elements of $G$ fixes more than 2 vertices.
Suppose some order 2 element of $G$ fixes 2 vertices of $\Gamma$.  Thus, by the Fixed Vertex Property, each order 2 element of $\widehat{G}$ fixes 2 vertices, and hence also pointwise fixes the edge between the 2 vertices.  Now observe that two distinct involutions of $\widehat{G}$ cannot pointwise fix the same edge, since a cyclic group can have at most one element of order 2. 

 Observe that $\widehat{G}$ contains a subgroup $\widehat{H}\cong D_2$. Since $D_2$ contains 3 elements of order 2, $\Gamma$ has 3 edges which are each pointwise fixed by precisely one order 2 element of $\widehat{G}$.  We see as follows that each order 2 element of $\widehat{H}$ must setwise fix all 3 of these edges. Let $\widehat{g}$ and $\widehat{h}$ be order 2 elements of $\widehat{H}$, and let $x$ and $y$ be the vertices of the edge that is pointwise fixed by $\widehat{g}$. Since all elements of $D_2$ commute, $g(h(x)) = h(g(x)) = h(x)$, so $h(x)$ is fixed by $g$. Since $x$ and $y$ are the only vertices that are fixed by $g$, $h(x) \in \{x,y\}$. Similarly for $h(y)$. So $\widehat{h}$ setwise fixes the edge $\overline{xy}$. It follows that each order 2 element of $\widehat{H}$ setwise fixes all 3 of these edges.  This implies that each order 2 element fixes the midpoint of each of the 3 edges.  These 3 midpoints determine a geodesic, which must be pointwise fixed by all 3 order 2 elements of $\widehat{H}$. But this is impossible since a cyclic group can have at most one element of order 2.  \end{proof}

\medskip

\begin{lemma}  \label{involution} Let $G\leq \Aut(K_m)$ which is isomorphic to $A_4$.  Suppose there is an embedding $\Gamma$ of $K_{m}$ in $S^3$ such that $G$ is induced on $\Gamma$ by an isomorphic subgroup $\widehat{G}\leq\mathrm{SO}(4)$.   If an order 2 element of $G$ fixes some vertex $v$, then $v$ is fixed by every element of $G$.
\end{lemma}

\begin{proof}  Suppose an order 2 element $\varphi_1\in G$ fixes a vertex $v$.  By Lemma \ref{n2not2}, $\varphi_1$ fixes no other vertices of $\Gamma$.  Since $G\cong A_4$, there is an involution  $\varphi_2\in G$ such that $\gen{\varphi_1,\varphi_2}\cong \Z_2\times\Z_2$.  Now by the Orbits Lemma, $\varphi_2$ takes fixed vertices of $\varphi_1$ to fixed vertices of $\varphi_1$.  Thus $\varphi_2(v)=v$.  Hence $v$ is fixed by $\gen{\varphi_1,\varphi_2}$.  Furthermore, all of the order 2 elements of $G$ are in $\gen{\varphi_1,\varphi_2}$.  Thus $v$ is the only vertex fixed by any order 2 element of $G$.

Let $\psi$ be an order 3 element of $G$.  Now $\psi\varphi_1\psi^{-1}$ has order 2 and fixes $\psi(v)$.  Thus $\psi(v)=v$.  Since $G=\gen{ \varphi_1,\varphi_2, \psi}$, $v$ is fixed by every element of $G$.
\end{proof}

\medskip

\begin{lemma}  \label{n2n3}  Let $G\leq \Aut(K_m)$ which is isomorphic to $A_4$.  Suppose there is an embedding $\Gamma$ of $K_{m}$ in $S^3$ such that $G$ is induced on $\Gamma$ by an isomorphic subgroup $\widehat{G}\leq\mathrm{SO}(4)$.    If some order 2 element of $G$ fixes a vertex of $\Gamma$, then no element of $G$ fixes 3 vertices.\end{lemma}

\begin{proof} Suppose some order 2 element of $G$ fixes a vertex $v$.  By Lemma 2, every element of $G$ fixes $v$. Suppose that $G$ contains an  element $\psi$ which fixes 3 vertices.   It follows from Lemma 1 that the order of $\psi$ must be  3.  Now let $g\in G$ have order $3$ such that $\gen{ g, \psi}$ is not cyclic.  It follows from the Fixed Vertex property that $\fix(\widehat{\psi})$ and $\fix(\widehat{g})$ each consist of 3 vertices and 3 edges.  Since $v\in \fix(\widehat{g})\cap\fix(\widehat{\psi})$, there must be another point $x\in \fix(\widehat{g})\cap\fix(\widehat{\psi})$.  However, since two edges cannot intersect in their interiors, $x$ must be a vertex of $\Gamma$.  This implies that  $\widehat{\psi}$ and $\widehat{g}$ pointwise fix the edge $\overline{xv}$.  However, this is impossible since $\gen{ \psi,g}$ is not cyclic.  Thus no element of $G$ fixes 3 vertices.\end{proof}
 \bigskip

 \begin{lemma}  \label{nknot3} Let $G\leq \Aut(K_m)$ which is isomorphic to  $A_5$.  Suppose there an embedding $\Gamma$ of $K_{m}$ in $S^3$ such that $G$ is induced on $\Gamma$ by an isomorphic subgroup $\widehat{G}\leq\mathrm{SO}(4)$.    Then no element of $G$ fixes 3 vertices.\end{lemma}

\begin{proof}  Recall that the only even order elements of $A_5$ are involutions.   By Lemma 1, no involution of $G$ fixes more than 1 vertex.  Let $\psi$ be an element of $G$ of odd order $q$ and suppose that $\psi$ fixes 3 vertices. Now $G$ contains an involution $\varphi$ such that $\gen{ \varphi, \psi}\cong D_q$.   Thus for every vertex $x$ which is fixed by $\psi$, $\psi\varphi(x)=\varphi\psi^{-1}(x)=\varphi(x)$.  Hence $\varphi(x)$ is also fixed by $\psi$.  So $\varphi$ setwise fixes the set of fixed vertices of $\psi$.  Since $\psi$ fixes 3 vertices and $\varphi$ has order 2, $\varphi$ must fix one of these 3 vertices $v$. 

 Let $H\leq G$ such that $H\cong A_4$ and $H$ contains the involution $\varphi$. Then by Lemma \ref{involution}, every element of $H$ fixes $v$.  Since $\varphi$ fixes $v$ and $\psi$ fixes 3 vertices, it follows from Lemma \ref{n2n3} that $\psi\not\in H$. Therefore $\gen{ \psi,H}=G$,  because $A_5$ has no proper subgroup containing $A_4$ as a proper subgroup.  Hence every element of $G$ fixes $v$.  Now let $g\in G$ have order $q$ such that $\gen{ g, \psi}$ is not cyclic.  By the Fixed Vertex Property, $\fix(\widehat{g})$ and $\fix(\widehat{\psi})$ each contain 3 vertices and 3 edges.  Thus we can repeat the argument given in the proof of Lemma \ref{n2n3} to get a contradiction.\end{proof}

\medskip

\begin{lemma}\label{1fixed}
 Let $G\leq \Aut(K_m)$ which is isomorphic to $A_5$.  Suppose there is an embedding $\Gamma$ of $K_{m}$ in $S^3$ such that $G$ is induced on $\Gamma$ by an isomorphic subgroup $\widehat{G}\leq\mathrm{SO}(4)$.  If an element $\psi\in G$ with odd order $q$ fixes precisely one vertex $v$, then $v$ is fixed by every element of $G$ and no other vertex is fixed by any non-trivial element of $G$.  
\end{lemma}

\begin{proof} There is an involution $\varphi\in G$ such that $\gen{ \varphi, \psi}\cong D_q$.  Now $\psi\varphi(v)=\varphi\psi^{-1}(v)=\varphi(v)$.  Since $v$ is the only vertex fixed by $\psi$, we must have $\varphi(v)=v$.  Now $G$ contains a subgroup $H\cong A_4$ containing $\varphi$.  By Lemma \ref{involution}, since $\varphi$ fixes $v$ every element of $H$ fixes $v$.  Since $A_4$ does not contain $D_3$ or $D_5$, $\psi\not \in H$.  Hence as in the proof of Lemma \ref{nknot3}, $\gen{ \psi, H}=G$.  Thus every element of $G$ fixes $v$.  Every involution in $G$ is an element of a subgroup isomorphic to $A_4$.  Thus by Lemma \ref{n2not2}, $v$ is the only vertex which is fixed by any involution in $G$.

Let $\beta\in G$ be of order $p=3$ or 5.  Suppose $\beta$ fixes some vertex $w\not =v$.  Thus all of the elements in $\gen{ \beta}\cong \Z_p$ fix $v$ and $w$.  Let $n$ denote the number of subgroups of $G$ that are isomorphic to $\Z_p$.  Thus $n=6$ or $n=10$ according to whether $p=5$ or $p=3$ respectively.  By the Fixed Vertex Property, all of the subgroups isomorphic to $\Z_p$ also fix 2 vertices.   If $g\in G$ fixes $w$, then $g$ pointwise fixes the edge $\overline{vw}$ and hence $\gen{ g, \beta}$ is cyclic.  It follows that each of the $n$ subgroups isomorphic to $\Z_p$ fixes a distinct vertex in addition to $v$. These $n$ vertices together with $v$ span a subgraph $\Lambda\subseteq \Gamma$ which is an embedding of $K_{n+1}$ such that $\Lambda$ is setwise invariant under $\widehat{G}$ and $\widehat{G}$ induces an isomorphic group action on $\Lambda$.  However, $n+1=7$ or 11.  Since $G\cong A_5$ contains a subgroup isomorphic to $D_2$, this contradicts the $D_2 $ Lemma.  Thus $v$ is the only vertex which is fixed by any order $p$ element of $G$.\end{proof}
\bigskip

 The following general result may be well known.  However, since we could not find a reference, we include an elementary proof here.  Observe that in contrast with Lemma \ref{L:dodecahedron}, if $\widehat{G}$ acts on $S^3$ as the orientation preserving isometries of a regular 4-simplex then the order 5 elements are glide rotations.

 \begin{lemma} \label{L:dodecahedron}
Suppose that  $\widehat{G}\leq \mathrm{SO}(4)$ such that $\widehat{G}\cong A_5$ and every order 5 element of $\widehat{G}$ is a rotation of $S^3$.  Then $\widehat{G}$ induces the group of rotations of a regular solid dodecahedron.
\end{lemma}

\begin{proof}  The group $\widehat{G}$ contains subgroups $J_1$, \dots, $J_6$ which are isomorphic to $\mathbb{Z}_5$ and involutions $\varphi_1$, \dots, $\varphi_6$ such that for each $i$, $H_i=\gen{ J_i,\varphi_i}\cong D_5$.   Now since every order 5 element of $\widehat{G}$ is a rotation of $S^3$, for each $i$ there is a geodesic circle $L_i$ which is pointwise fixed by every element of $J_i$.  Furthermore, because $H_i\cong D_5$, the circle $L_i$ must be inverted by the involution $\varphi_i$.  Hence there are points $p_i$ and $q_i$ on $L_i$ which are fixed by $\varphi_i$.  Now every involution in $H_i\cong D_5$ is conjugate to $\varphi_i$ by an element of $J_i$.  Hence every involution in $H_i$ also fixes both $p_i$ and $q_i$.   For each $i$, let $S_i$ denote the geodesic sphere $S_i$ which meets the circle $L_i$ orthogonally in the points $p_i$ and $q_i$.  Now $S_i$ is setwise invariant under every element of $H_i$.  

By analyzing the structure of $A_5$, we see that each involution in $H_1$ is also contained in precisely one of the groups $H_2$, \dots, $H_6$.   Thus for each $i\not =1$, the 2 points which are fixed by the involution in $H_1\cap H_i$ are contained in $S_1\cap S_i$.  Since $p_1$ is fixed by every involution in $H_1$, it follows that $p_1$ is contained in every $S_i$.  Observe that the set of geodesic spheres $\{S_1, \dots, S_6\}$ is setwise invariant under $\widehat{G}$.  Since $p_1$ is in every $S_i$, this implies that the orbit $P$ of $p_1$ is contained in every $S_i$.  

If $p_1$ is fixed by every element of $\widehat{G}$, then $\widehat{G}$ induces the group of rotations of a regular solid dodecahedron centered at $p_1$.  Thus we assume that $p_1$ is not fixed by every element of $\widehat{G}$.  Since $\widehat{G}$ can be generated by elements of order 5, it follows that some order 5 element of $\widehat{G}$ does not fix $p_i$.  The orbit of $p_1$ under that  element must contain at least 5 elements, and hence $|P|\geq 5$.   Suppose that some $S_i\not =S_1$.  Then $S_1\cap S_i$ consists of a geodesic circle $C$ containing the set $P$.  Since $|P|>2$, the circle $C$ is uniquely determined by $P$.   

Now $C$ must be setwise invariant under $\widehat{G}$ since $P$ is.  Thus the core $D$ of the open solid torus $S^3-C$ is also setwise invariant under $\widehat{G}$.  Since a pair of circles cannot be pointwise fixed by a non-trivial orientation preserving isometry of $S^3$, $\widehat{G}$ induces a faithful action of $C\cup D$ taking each circle to itself.  But the only finite groups that can act faithfully on a circle are cyclic or dihedral, and $A_5$ is not the product of two such groups.  Thus every $S_i=S_1$.  

Recall that for each $i$, the geodesic circle $L_i$ is orthogonal to the sphere $S_i$ and is pointwise fixed by every element of $J_i$.  Since all of the geodesic circles $L_1$, \dots, $L_6$ are orthogonal to the single sphere $S_i=S_1$, they must all meet at a point $x$ in a ball bounded by $S_1$.  Now $\widehat{G}=\gen{ J_1, J_2}$, and every element of $J_1$ and $J_2$ fixes $x$.  Thus $\widehat{G}$ fixes the point $x$.  Hence again $\widehat{G}$ induces the group of rotations of a solid dodecahedron centered at the point $x$.\end{proof}

\bigskip

 Suppose that $G$ is a group acting faithfully on $K_m$.  Let  $V$ denote the vertices of $K_m$ and let $|\mathrm{fix}(g|V)|$ denote the number of vertices of $K_m$ which are fixed by $g\in G$.  Burnside's Lemma \cite{Bu} gives us the following equation:

$${\rm \#\  vertex\  orbits} = \frac{1}{|G|}\sum_{g \in G}{|\mathrm{fix}(g|V)|}$$

\bigskip

 We shall use the fact that the left side of this equation is an integer to prove the necessity of our conditions for $K_m$ to have an embedding $\Gamma$ such that $G=\TSG(\Gamma)$ is isomorphic to $A_4$ or $A_5$. By the Fixed Vertex Property, all elements of the same order fix the same number of vertices of $\Gamma$.  So we will use $n_k$ to denote the number of fixed vertices of an element of $G$ of order $k$.  Observe that $n_1$ is always equal to $m$.
 
 \medskip
 
 \begin{thm} \label{A4necessary} If a complete graph $K_m$ has an embedding $\Gamma$ in $S^3$ such that $\TSG(\Gamma)\cong A_4$, then $m\equiv 0$, $1$, $4$, $5$, $8$ $\pmod {12}$.\end{thm}
 
 \begin{proof} Let $G=\TSG(\Gamma)\cong A_4$.  By applying the Isometry Theorem, we obtain a re-embedding $\Lambda$ of $K_m$ such that $G$ is induced on $\Lambda$ by an isomorphic subgroup $\widehat{G} \leq \mathrm{SO}(4)$.  Thus we can apply our lemmas.   Note that $\vert A_4 \vert = 12$, and $A_4$ contains 3 order 2 elements and 8 order 3 elements.  Thus Burnside's Lemma tells us that $\frac{1}{12}(m + 3n_2+8n_3)$ is an integer.

By Lemma 1, we know that $n_2$ = 0 or 1, and by Lemma 3 we know that if $n_2=1$ then $n_3\not= 3$.  Also, by Lemma 2, if $n_3 = 0$, then $n_2 = 0$.  So there are 5 cases, summarized in the table below.  In each case, the value of $m \pmod{12}$ is determined by knowing that $\frac{1}{12}(m + 3n_2+8n_3)$ is an integer.
\bigskip

\begin{center}
\begin{tabular}{|c|c|c|}
	\hline
	$n_2$ & $n_3$ & $m \pmod{12}$ \\ \hline
	0 & 0 or 3 & 0 \\
	0 & 1 & 4 \\
	0 & 2 & 8 \\
	1 & 1 & 1 \\
	1 & 2 & 5 \\ \hline
\end{tabular}
\end{center}\end{proof}

 \bigskip

 \begin{thm} \label{A5necessary} If a complete graph $K_m$ has an embedding $\Gamma $ in $S^3$ such that $\TSG(\Gamma)\cong A_5$, then $m\equiv 0$, $1$, $5$, $20 \pmod {60}$.\end{thm}

 \begin{proof} Let $G=\TSG(\Gamma)\cong A_5$.  By applying the Isometry Theorem, we obtain a re-embedding $\Lambda$ of $K_m$ such that $G$ is induced on $\Lambda$ by an isomorphic subgroup $\widehat{G} \leq \mathrm{SO}(4)$.  Note that $|A_5|=60$, and $A_5$ contains 15 elements of order 2, 20 elements of order 3, and 24 elements of order 5.  Thus Burnside's Lemma tells us that $ \frac{1}{60}(m+15n_2+20n_3+24n_5)$ is an integer.

By Lemma 4, for every $k>1$, $n_k<3$.  Every element of $G$ of order 2 or 3 is contained in some subgroup isomorphic to $A_4$.  Thus as in the proof of Theorem 1, we see that $n_2$ = 0 or 1, and if $n_3 = 0$ then $n_2 = 0$.  Also, by Lemma 5, if either $n_3 = 1$ or $n_5 = 1$, then all of $n_2$, $n_3$ and $n_5$ are 1. 

Suppose that $n_5=2$.  Then each order 5 element of $\widehat{G}$ must be a rotation.  Let $\widehat{\psi}$, $\widehat{\varphi}\in\widehat{G}$ such that $\widehat{\psi}$ has order 5, $\widehat{\varphi}$ has order 2, and $\gen{ \widehat{\psi},\widehat{\varphi}}\cong D_5$. Then there is a circle which is fixed pointwise by $\widehat{\psi}$ and inverted by $\widehat{\varphi}$.  Thus $\fix(\widehat{\varphi})$ intersects $\fix(\widehat{\psi})$ in 2 precisely points.  By Lemma \ref{L:dodecahedron}, we know that $\widehat{G}$ induces the group of rotations on a solid dodecahedron. Hence the fixed point sets of all of the elements of $\widehat{G}$ meet in two points, which are the points $\fix(\widehat{\varphi})\cap\fix(\widehat{\psi})$.  Now since $n_5=2$, $\fix(\widehat{\psi})$ contains precisely 2 vertices, and hence an edge $e$.  Thus $e$ must be inverted by $\widehat{\varphi}$. It follows that the midpoint of $e$ is  one of the two fixed points of $\widehat{G}$. Since $G$ is not a dihedral group we know that $e$ is not setwise invariant under every element of $\widehat{G}$.  Thus there are other edges in the orbit of $e$ which intersect $e$ in its midpoint.  Since two edges cannot intersect in their interiors, we conclude that $n_5\not =2$.

 There are 4 cases summarized in the table below.

\bigskip

\begin{center}
\begin{tabular}{|c|c|c|c|}
	\hline
	$n_2$ & $n_3$ & $n_5$ & $m \pmod{60}$ \\ \hline
	0 & 0 & 0 & 0 \\
	0 & 2 & 0 & 20 \\
	1 & 1 & 1 & 1 \\
	1 & 2 & 0 & 5 \\ \hline
\end{tabular}
\end{center}\end{proof}
\bigskip

\begin{thm}\label{S4necessary} If a complete graph $K_m$ has an embedding $\Gamma$ in $S^3$ such that $\TSG(\Gamma)\cong S_4$, then $m\equiv 0$, $4$, $8$, $12$, $20\pmod {24}$.
\end{thm}

 \begin{proof} Let $G=\TSG(\Gamma)\cong S_4$.  By applying the Isometry Theorem, we obtain a re-embedding $\Lambda$ of $K_m$ such that $G$ is induced on $\Lambda$ by an isomorphic subgroup $\widehat{G} \leq \mathrm{SO}(4)$.  Suppose that some order 4 element $\widehat{g}\in\widehat{G}$ has non-empty fixed point set.  Then $\fix(\widehat{g})\cong S^1$.  Thus $\fix (\widehat{g}^2)=\fix(\widehat{g})$.  Let $(v_1,v_2,v_3,v_4)$ be a 4-cycle of vertices under $g$.  Then $g^2$ inverts the edges $\overline{v_1v_3}$ and $\overline{v_2v_4}$.  Thus  $\fix (\widehat{g}^2)$ intersects both $\overline{v_1v_3}$ and $\overline{v_2v_4}$.  Hence $\widehat{g}$ fixes a point on each of $\overline{v_1v_3}$ and $\overline{v_2v_4}$.  But this is impossible since $(v_1,v_2,v_3,v_4)$ is induced by $g$.  Thus every order 4 element of $\widehat{G}$ has empty fixed point set.  In particular, no order 4 element of $G$ fixes any vertices of $\Gamma$.  Thus $m\not \equiv 1\pmod 4$.  Since $A_4\leq S_4$, by Theorem \ref{A4necessary}, $m\equiv 0$, $1\pmod 4$.  It follows that $m\equiv 0\pmod{4}$.

 Suppose that $m=24n+16$ for some $n$.  The group $S_4$ has 3 elements of order 2 which are contained in $A_{4}$, 6 elements of order 2 which are not contained in $A_4$, 8 elements of order 3, and 6 elements of order 4.  By the Fixed Vertex Property, each of the elements of any one of these types fixes the same number of vertices.  So according to Burnside's Lemma $\frac{1}{24}((24n+16) + 3n_2 + 6n_2' + 8n_3 + 6n_4)$ is an integer, where $n_2$ denotes the number of fixed vertices of elements of order 2 which are contained in $A_4$ and $n_2'$ denotes  the number of fixed vertices of elements of order 2 which are not contained in $A_4$.  
 
 We saw above that $n_4=0$.  By Lemma \ref{n2not2}, $n_2=0$ or 1.  However, since $n_2$ is the only term with an odd coefficient, we cannot have $n_2=1$.  Also since the number of vertices $m=24n+16\equiv 1\pmod{3}$, each element of order 3 must fix one vertex.   Thus $n_2=0$ and $n_3=1$.  Hence $\frac{1}{24}(16 + 6n_2' + 8)$ is an integer.  It follows that $n_2'=0$, since $n_2'\not >3$.  Let $\psi$ be an order 3 element of $\widehat{G}$.  Since $n_3=1$, $\psi$ must be a rotation about a circle $L$ containing a single vertex $v$.  Since $\widehat{G}\cong S_4$, there is an involution $\varphi\in \widehat{G}$ such that $\gen{ \psi, \varphi}\cong D_3$.  It follows that $\varphi$ inverts $L$.  However, since $v$ is the only vertex on $L$, $\varphi(v)=v$.  This is impossible since $n_2=n_2'=0$.  Thus $m\not\equiv 16\pmod {24}$.  The result follows.\end{proof}
\bigskip

\section{Embedding Lemmas}

For a given $n$, we would like to be able to construct an embedding of $K_m$ which has a particular topological symmetry group.  We do this by first embedding the vertices of $K_m$ so that they are setwise invariant under a particular group of isometries, and then we embed the edges of $K_m$ using the results below.  Note that Lemma \ref{L:edge} applies to any finite  group $G$ of diffeomorphisms of $S^3$, regardless of whether the diffeomorphisms in $G$ are orientation reversing or preserving.

\begin{lemma}\label{L:edge}
Let $G$ be a finite group of diffeomorphisms of $S^3$ and let $\gamma$ be a graph whose vertices are embedded in $S^3$ as a set $V$ such that $G$ induces a faithful action on $\gamma$. Let $Y$ denote the union of the fixed point sets of all of the non-trivial elements of $G$.  Suppose that adjacent pairs of vertices in $V$ satisfy the following hypotheses:

\begin{enumerate}

\item
 If a pair is pointwise fixed by non-trivial elements $h$, $g\in G$, then $\fix(h)=\fix(g)$.

\item
No pair is interchanged by an element of $G$.

\item Any pair that is pointwise fixed by a non-trivial $g\in G$ bounds an arc in $\fix(g)$ whose interior is disjoint from $V\cup (Y-\fix(g))$.

\item Every pair is contained in a single component of $S^3-Y$.

\end{enumerate}

Then there is an embedding of the edges of $\gamma$ such that the resulting embedding of $\gamma$ is setwise invariant under $G$.  
\end{lemma}

\begin{proof}  We partition the edges of $\g$ into sets
$F_1$ and $F_2$,
where $F_1$ consists of all edges of $\g$
both of whose embedded vertices are fixed by some non-trivial element of $G$,
and $F_2$ consists of the remaining edges of $\g$.  Thus, each $F_i$ is setwise invariant under $G$.

We first embed the edges in $F_1$ as follows.
Let $\lbrace f_{1},\dots ,f_{m}\rbrace $ be a set of edges
consisting of one representative
from the orbit of each edge in $F_1$.
Thus for each $i$, some non-trivial $g_i \in G$ fixes the embedded vertices of $f_i$.  Furthermore, by hypothesis (1), $\fix(g_i)$ is uniquely determined by $f_i$.
By hypothesis (3), the vertices of $f_i$ bound an arc $A_i\subseteq\fix(g_i)$ whose interior is disjoint from $V$ and from the fixed point set of any element of $G$ whose fixed point set is not $\fix(g_i)$.  We embed the edge $f_i$ as the arc $A_i$.  
 Now it follows from our choice of $A_i$ that the interiors of the arcs in the orbits of $A_1$, \dots, $A_m$ are pairwise disjoint.

Now let $f$ be an edge in $F_1-\lbrace f_{1},\dots ,f_{m}\rbrace$.
Then for some $g\in G$ and some edge $f_i$, we have $g(f_i)=f$.
We embed the edge $f$ as $g(A_i)$.
To see that this is well-defined,
suppose that for some $h\in G$ and some $f_j$, we also have $f=h(f_j)$.
Then $i = j$,
since we picked only one representative from each edge orbit.
Therefore $g(f_i) = h(f_i)$.
This implies $h\inv g$ fixes both vertices of $f_i$
since by hypothesis~(2)
no edge of $\g$ is inverted by $G$.
Now, by hypothesis~(1), if $h^{-1}g$ is non-trivial, then
$\fix(h\inv g) = \fix(g_{i})$.
Since $A_i \subseteq \fix(g_{i})$,
it follows that $h(A_i) = g(A_i)$, as desired.  We can thus unambiguously embed all of the edges of $F_1$.  Let $E_1$ denote this set of embedded edges.
By our construction, $E_1$ is setwise invariant under $G$.

Next we will embed the edges of $F_2$.    Let $\pi :S^{3}\rightarrow S^{3}/G$ denote the
quotient map.    Then $\pi|(S^3-Y) $ is a covering
map, and the quotient space $Q=(S^{3}-Y)/G$ is a 
$3$-manifold.  We will embed representatives of the edges of $F_2$ in the quotient space $Q$, and then lift them to get an embedding of the edges in $S^3$.

Let $\{ e_{1},\dots , e_{n}\}$ be a set of edges
consisting of one representative
from the orbit of each edge in $F_2$.
For each $i$, let $x_{i}$ and $y_{i}$ be the embedded vertices of $e_i$ in $V$.
By hypothesis (4), for each $i = 1$, \dots, $n$,
there exists a path $\alpha _{i}$ in $S^3$
from $x_{i}$ to $y_{i}$
whose interior is disjoint from $V \cup Y$.  Let
$\alpha _{i}'=\pi \circ \alpha _{i}$.
Then $\alpha _{i}'$ is a path or loop
from $\pi (x_{i})$ to $\pi (y_{i})$
whose interior is in $Q$.
Using general position in $Q$,
we can homotop each $\alpha _{i}'$, fixing its endpoints,
to a simple path or loop $\rho _{i}'$
such that the interiors of the $\rho _{i}'(I)$ are pairwise disjoint
and are each disjoint from $\pi (V \cup Y)$.
Now, for each $i$,
we lift $\rho _{i}'$ to a path $\rho_i$ beginning at $x_{i}$.  Then each $\mathrm{int}(\rho_i)$ is disjoint from $V \cup Y$.
Since $\rho _{i}'=\pi \circ \rho _{i}$
is one-to-one except possibly on the set $\lbrace 0,1\rbrace $,
$\rho _{i}$ must also be one-to-one except possibly on the set
$\lbrace 0,1\rbrace $.
Also, since $\rho _{i}'$ is homotopic fixing its endpoints to $\alpha _{i}'$,
$\rho _{i}$ is homotopic fixing its endpoints to $\alpha _{i}$.
In particular, $\rho _{i}$ is a simple path
from $x_{i}$ to $y_{i}$.
We embed the edge $e_{i}$ as $\rho _{i}(I)$.

Next, we will embed an arbitrary edge $e$ of $F_2$.
By hypothesis (2) and the definition of $F_2$,
no edge in $F_2$ is
setwise invariant under any non-trivial element of $G$.
Hence there is a unique $g \in G$ and a unique $i\leq n$
such that $e=g'(e_i)$.
It follows that $e$ determines a unique arc $g(\rho_i(I))$
between $g(x_i)$ and $g(y_i)$.
We embed $e$ as $g(\rho _{i}(I))$.  By the uniqueness of $g$ and $i$, this embedding is well-defined.
Let $E_2$ denote the set of embedded edges
of $F_2$.
Then $G$ leaves $E_2$ setwise invariant.  

Now, since each $\mathrm{int}(\rho _{i}'(I))$
is disjoint from $\pi (V)$,
the interior of each embedded edge of $E_2$
is disjoint from $V$.
Similarly, since
$\rho'_{i}(I)$ and $\rho '_{j}(I)$ have disjoint
interiors when $i\mathbin{\not =}j$, for every $g$, $h\in G$, $g(\rho
_{i}(I))$ and $h(\rho _{j}(I))$ also have disjoint interiors when $i\mathbin{\not =}j$.
And since $\rho_i'$ is a simple path or loop
whose interior is disjoint from $\pi(Y)$,
if $g\mathbin{\not =}h$,
then $g(\rho _{i}(I))$ and $h(\rho _{i}(I))$
have disjoint interiors.  Thus the embedded edges of $E_2$ have pairwise disjoint interiors.

Let $\Gamma $ consist of the set of embedded vertices $V$
together with the set of embedded edges $E_1 \cup E_2$.
Then $\Gamma $ is setwise invariant under $G$.
Also, every edge in $E_1$ is an arc of $Y$, whose interior is disjoint from $V$,
and the interior of every edge in $E_2$ is a subset of $S^3 - (Y\cup V)$.  Therefore the interiors of the edges in $E_1$ and $E_2$ are disjoint.  Hence $\Gamma$ is an embedded graph with
underlying abstract graph $\gamma $, and
$\Gamma$ is setwise invariant under $G$. 
\end{proof}

\medskip

 We use Lemma 7 to prove the following result.  Note that $\Diff(S^3)$ denotes the group of orientation preserving diffeomorphisms of $S^3$.  Thus by contrast with Lemma \ref{L:edge}, the Edge Embedding Lemma only applies to finite groups of orientation preserving diffeomorphisms of $S^3$.
\medskip

\begin{Graph}  Let $G$ be a finite subgroup of $\Diff(S^3)$ and let $\gamma$ be a graph whose vertices are embedded in $S^3$ as a set $V$ such that $G$ induces a faithful action on $\gamma$. Suppose that adjacent pairs of vertices in $V$ satisfy the following hypotheses:
\begin{enumerate}

\item If a pair is pointwise fixed by non-trivial elements $h$, $g\in G$, then $\fix(h)=\fix(g)$.

\item For each pair $\{v, w\}$ in the fixed point set $C$ of some non-trivial element of $G$, there is an arc $A_{vw}\subseteq C$ bounded by $\{v, w\}$ whose interior is disjoint from $V$ and from any other such arc $A_{v'w'}$.

\item If a point in the interior of some $A_{vw}$ or a pair $\{v,w\}$ bounding some $A_{vw}$ is setwise invariant under an $f\in G$, then $f(A_{vw})=A_{vw}$.

\item  If a pair is interchanged by some $g\in G$, then the subgraph of $\gamma$ whose vertices are pointwise fixed by $g$ can be embedded in a proper subset of a circle.

\item  If a pair is interchanged by some $g\in G$, then $\fix(g)$ is non-empty, and for any $h\not =g$, then $\fix(h)\not=\fix(g)$.

\end{enumerate}

Then there is an embedding of the edges of $\gamma$ in $S^3$ such that the resulting embedding of $\gamma$ is setwise invariant under $G$.  
\end{Graph}

\begin{proof} Let $\gamma'$ denote $\gamma$ together with a valence 2 vertex added to the interior of every edge whose vertices are interchanged by some element of $G$.  Then $G$ induces a faithful action on $\gamma'$, since $G$ induces a faithful action on $\gamma$.  For each $g\in G$ we will let $g' $ denote the automorphism of $\gamma'$ induced by $g$, and let $G'$ denote the group of automorphisms of $\gamma'$ induced by $G$.  No element of $G'$ interchanges a pair of adjacent vertices of $\gamma'$.  Since $G$ induces a faithful action on $\gamma'$, each $g'\in G$ is induced by a unique $g\in G$.

Let $M$ denote the set of vertices of $\gamma'$ which are not in $\gamma$.  Each vertex $m\in M$ is fixed by an element $f'\in G'$ which interchanges the pair of vertices adjacent to $m$.  We partition $M$ into sets $M_1$ and $M_2$, where $M_1$ contains those vertices of $M$ whose adjacent vertices are both fixed by a non-trivial automorphism in $G'$ and $M_2$ contains those vertices of $M$ whose adjacent vertices are not both fixed by a non-trivial automorphism in $G'$.

We first embed the vertices of $M_1$.  Let $\lbrace m_{1},\dots ,m_{r}\rbrace $ be a set consisting of one representative
from the orbit of each vertex in $M_1$, and for each $m_i$, let $v_i$ and $w_i$ denote the vertices which are adjacent to the vertex $m_i$ in $\gamma'$.  Thus $v_i$ and $w_i$ are adjacent vertices of $\gamma$.  By definition of $M_1$, each $m_i$ is fixed both by some automorphism $f_i'\in G'$ which interchanges $v_i$ and $w_i$ and by some element $h_i'\in G'$ which fixes both $v_i$ and $w_i$.  Let $f_i$ and $h_i$ be the elements of $G$ which induce $f_i'$ and $h_i'$ respectively.  Let $A_{v_iw_i}$ denote the arc in $\fix(h_i)$ given by hypothesis (2).  Since $f_i$ interchanges $v_i$ and $w_i$, it follows from hypothesis (3) that $f_i(A_{v_iw_i})=A_{v_iw_i}$.  Also since $f_i$ has finite order, there is a unique point $x_i$ in the interior of $A_{v_iw_i}$ which is fixed by $f_i$.  We embed $m_i$ as the point $x_i$.  By hypothesis (2), if $i\not =j$, then $A_{v_iw_i}$ and $A_{v_jw_j}$ have disjoint interiors, and hence $x_i\not =x_j$.

We see as follows that the choice of $x_i$ does not depend on the choice of either $h_i$ or $f_i$.  Suppose that $m_i$ is fixed by some $f'\in G'$ which interchanges $v_i$ and $w_i$ and some $h'\in G'$ which fixes both $v_i$ and $w_i$.  Let $f$ and $h$ be the elements of $G$ which induce $f'$ and $h'$ respectively. Since both $f$ and $h$ leave the pair $\{v_i, w_i\}$ setwise invariant, by hypothesis (3) both $f$ and $h$ leave the arc $A_{v_iw_i}$ setwise invariant.   Since $h$ has finite order and fixes both $v_i$ and $w_i$, $h$ pointwise fixes the arc $A_{v_iw_i}$, and hence $\fix(h)=\fix(h_i)$.  Thus the choice of the arc $A_{v_iw_i}$ does not depend on $h$.  Also since $G$ has finite order, and $f$ and $f_i$ both interchange $v_i$ and $w_i$ leaving $A_{v_iw_i}$ setwise invariant, $f^{-1}f_i$ pointwise fixes $A_{v_iw_i}$. Hence $f|A_{v_iw_i}=f_i|A_{v_iw_i}$, and thus the choice of $x_i$ is indeed independent of $f_i$ and $h_i$.  In fact, by the same argument we see that $x_i$ is the unique point in the interior of $A_{v_iw_i}$ that is fixed by an element of $G$ which setwise but not pointwise fixes $A_{v_iw_i}$ (we will repeatedly use this fact below). 

Now let $m$ denote an arbitrary point in $M_1$.  Then for some $i$ and some automorphism $g'\in G'$, $m=g'(m_i)$.  Let $v$ and $w$ be the vertices that are adjacent to the vertex $m$ in $\gamma'$.   Then $\{v,w\}=g'(\{v_i,w_i\})$.  Let $g$ be the element of $G$ which induces $g'$.  We embed $m$ as the point $g(x_i)$.  To see that this embedding is unambiguous, suppose that for some other automorphism $\varphi'\in G'$, we also have $m=\varphi'(m_j)$.  Then $j=i$, since the orbits of $m_{1}$,\dots, $m_{r}$ are disjoint.  Let $\varphi$ denote the element of $G$ which induces $\varphi'$.  Then $g^{-1}\varphi(m_i)=m_i$, and hence $g^{-1}\varphi(\{v_i,w_i\})=\{v_i,w_i\}$.  It follows from hypothesis (3) that $g^{-1}\varphi(A_{v_iw_i})=A_{v_iw_i}$.  Now $x_i$ is the unique point in the interior of $A_{v_iw_i}$ that is fixed by an element of $G$ which setwise but not pointwise fixes $A_{v_iw_i}$.  It follows that $\varphi( x_i)=g(x_i)$. Thus our embedding is well defined for all of the points of $M_1$.  Furthermore, since $m_1$, \dots, $m_r$ have distinct orbits under $G'$, the pairs $\{v_1,w_1\}$, \dots, $\{v_r,w_r\}$ have distinct orbits under $G$.  Hence the arcs $A_{v_1w_1}$,\dots, $A_{v_rw_r}$ are not only disjoint, they also have distinct orbits under $G$.  Thus the points of $M_1$ are embedded as distinct points of $S^3$.

Next we embed the vertices of $M_2$.  Let $\{g_1',\dots,g_q'\}$ consist of one representative from each conjugacy class of automorphisms in $G'$ which fix a point in $M_2$, and let each $g_i'$ be induced by $g_i\in G$.  For each $g_i'$, from the set of vertices of $M_2$ that are fixed by that $g_i'$, choose a subset $\{p_{i1},\dots, p_{in_i}\}$  consisting of one representative from each of their orbits under $G$. 

 Let $F_i=\fix(g_i)$.  By hypothesis (5) and Smith Theory \cite{Sm}, $F_i\cong S^1$ and $F_i$ is not the fixed point set of any element of $G$ other than $g_i$.   Thus each arc $A_{vw}$ that is a subset of $F_i$ corresponds to some edge of $\gamma$ whose vertices are fixed by $g_i$.  By hypothesis (4) the subgraph of $\gamma$ whose vertices are fixed by $g_i$ is homeomorphic to a proper subset of a circle.  Furthermore, since $G\leq \Diff(S^3)$ is finite, the fixed point set of any non-trivial element of $G$ other than $g_i$ meets $F_i$ in either 0 or 2 points.  Thus we can choose an arc $A_{i}\subseteq F_i$ which does not intersect any $A_{vw}$, is disjoint from the fixed point set of any other non-trivial element of $G$, and is disjoint from its own image under any other non-trivial element of $G$.  Now we can choose a  set $\{y_{i1},\dots, y_{in_i}\}$ of distinct points in the arc $A_i$, and embed the set of vertices $\{p_{i1},\dots, p_{in_i}\}$ as the set $\{y_{i1},\dots, y_{in_i}\}$.   Observe that if some $p_{ij}$ were also fixed by a non-trivial automorphism $g'\in G'$ such that $g'\not =g_i'$, then either $g'$ or $g'g_i'$ would fix both vertices adjacent to $p_{ij}$, which is contrary to the definition of $M_2$.  Hence $g_i'$ is the unique non-trivial automorphism in $G'$ fixing $p_{ij}$.  Thus our embedding of $p_{ij}$ is well defined.

We embed an arbitrary point $p$ of $M_2$ as follows.  Choose $i$, $j$, and $g'\in G'$ such that $p=g'(p_{ij})$, and $g'$ is induced by a unique element $g\in G$.  Since $p_{ij}$ is embedded as a point $y_{ij}\in A_i\subseteq \fix(g_i)$, we embed $p$ as $g(y_{ij})$.  To see that this is well defined, suppose that for some automorphism $\varphi'\in G'$ we also have $p=\varphi'(p_{lk})$, and $\varphi'$ is induced by $\varphi\in G$.  Then $p_{ij}=p_{lk}$, since their orbits share the point $p$, and hence are equal.   Now $(g')^{-1}\varphi'(p_{ij})=p_{ij}$.  But since $p_{ij}\in M_2$, no non-trivial element of $G'$ other than $h_i'$ fixes $p_{ij}$.  Thus either $(g')^{-1}\varphi'=g_i'$ or $g'=\varphi'$.   In the former case, the diffeomorphism $g^{-1}\varphi=g_i$ fixes the point $y_{ij}$, since $y_{ij}\in \fix(g_i)$.  Hence in either case $g(y_{ij})=\varphi(y_{ij})$.  Thus our embedding is well defined for all points of $M_2$.  


Recall that, if $i\not= j$, then $g_i$ and $g_j$ are in distinct conjugacy classes of $G$.  Also, by hypothesis (5), $F_i$ is not fixed by any non-trivial element of $G$ other than $g_i$.   Now it follows that $F_i$ is not in the orbit of $F_j$, and hence the points of $M_2$ are embedded as distinct points of $S^3$.  Finally, since the points of $M_1$ are each embedded in an arc $A_{vw}$ and the points of $M_2$ are each embedded in an arc which is disjoint from any $A_{vw}$, the sets of vertices in $M_1$ and $M_2$ have disjoint embeddings.

Let $V'$ denote $V$ together with the embeddings of the points of $M$.  Thus we have embedded the vertices of $\gamma'$ in $S^3$. We check as follows that the hypotheses of Lemma 7 are satisfied for $V'$.  When we refer to a hypothesis of Lemma 7 we shall put an * after the number of the hypothesis to distinguish it from a hypothesis of the lemma we are proving.  We have defined $V'$ so that it is setwise invariant under $G$ and $G$ induces a faithful action on $\gamma'$.  Also, by the definition of $\gamma'$, hypothesis (2*)  is satisfied.  Since $G$ is a finite subgroup of $\Diff(S^3)$, the union $Y$ of all of the fixed point sets of the non-trivial elements of $G$ is a graph in $S^3$.  Thus $S^3-Y$ is connected, and hence hypothesis (4*) is satisfied.

To see that hypothesis (1*) is satisfied for $V'$, suppose a pair $\{x,y\}$ of adjacent vertices of $\gamma'$ are both fixed by non-trivial elements $h$, $g\in G$.  If the pair is in $V$, then they are adjacent in $\gamma$, and hence by hypothesis (1) we know that $\fix(h)=\fix(g)$.  Thus suppose that $x\in M$.  Then $x\in M_1$, since the vertices in $M_2$ are fixed by at most one non-trivial automorphism in $G'$.  Now without loss of generality, we can assume that $x$ is one of the $m_i\in M_1$ and $y$ is an adjacent vertex $v_i$.  Thus $x$ is embedded as $x_i\in \INT(A_{v_iw_i})$.  Since $h$ and $g$ both fix $x_i$, by hypothesis (3), both $h(A_{v_iw_i})=A_{v_iw_i}$ and $g(A_{v_iw_i})=A_{v_iw_i}$.  It follows that $h$ and $g$ both fix $w_i$, since we know they fix $v_i$.  Now $\{v_i,w_i\}$ are an adjacent pair in $\gamma$.  Hence again by hypothesis (1), $\fix(h)=\fix(g)$.  It follows that hypothesis (1*) is satisfied for $V'$.  

It remains to check that hypothesis (3*) is satisfied for $V'$.  Let $s$ and $t$ be adjacent vertices of $\gamma'$ which are fixed by some non-trivial $g\in G$. We will show that $s$ and $t$ bound an arc in $\fix(g)$ whose interior is disjoint from $V'$, and if any $f\in G$ fixes a point in the interior of this arc then $\fix(f)=\fix(g)$.  First suppose that $s$ and $t$ are both in $V$.  Then no element of $G$ interchanges $s$ and $t$.  Now, by hypothesis (2), $s$ and $t$ bound an arc $A_{st}\subseteq \fix(g)$ whose interior is disjoint from $V$.  Furthermore, by (2), $\INT(A_{st})$ is disjoint from any other $A_{vw}$, if $\{v,w\}\not =\{s,t\}$.  Thus $\INT(A_{st})$ is disjoint from the embedded points of $M_1$.  Also the embedded points of $M_2$ are disjoint from any $A_{vw}$.  Thus $\INT(A_{st})$ is disjoint from $V'$.  Suppose some $f\in G$ fixes a point in $\INT(A_{st})$.    Then by hypothesis (3), $f(A_{st})=A_{st}$.  So $f$ either fixes or interchanges $s$ and $t$.  In the latter case, $s$ and $t$ would not be adjacent in $\gamma'$.   Thus $f$ fixes both $s$ and $t$.  Now by hypothesis (1), we must have $\fix(f)=\fix(g)$.  So the pair of vertices $s$ and $t$ satisfy hypothesis (3*).  

Next suppose that $s\in V'- V$.  Since $s$ and $t$ are adjacent in $\gamma'$, we must have $t\in V$.  Now $s$ is the embedding of some $m\in M$, and $m$ is adjacent to vertices $v$, $t\in V$ which are both fixed by $g$.  Thus $m\in M_1$.  By our embedding of $M_1$, for some $h\in G$ and some $i$, we have $s=h(x_i)$ where $x_i$ is adjacent to vertices $v_i$ and $w_i$ in $\gamma'$.  It follows that $\{v,t\}=h(\{v_i,w_i\})$. Thus $h^{-1}gh$ fixes $v_i$ and $w_i$.  Let $h_i$ and $A_{v_iw_i}\subseteq \fix(h_i)$ be as in the description of our embedding of the points in $M_1$.  Thus $v_i$ and $w_i$ are adjacent vertices of $\gamma$ which are fixed by both $h_i$ and $h^{-1}gh$.  By hypothesis (1), $\fix(h_i)=\fix(h^{-1}gh)$.  Thus $A_{v_iw_i}\subseteq \fix(h^{-1}gh)$.  Let $A=h(A_{v_iw_i})$.  Then $A$ is an arc bounded by $v$ and $t$ which is contained in $\fix(g)$.  Furthermore, the interior of $A_{v_iw_i}$ is disjoint from  $V$ and $x_i$ is the unique point in the interior of $A_{v_iw_i}$ that is fixed by an element of $G$ which setwise but not pointwise fixes $A_{v_iw_i}$.  Thus the interior of $A$ is disjoint from $V$ and $s=h(x_i)$ is the unique point in the interior of $A$ that is fixed by an element of $G$ which setwise but not pointwise fixes $A$.  Let $A_{st}$ denote the subarc of $A$ with endpoints $s$ and $t$.  Then $\INT(A_{st})$ is disjoint from $V'$, and $A_{st}$ satisfies hypothesis (3*).

Thus we can apply Lemma 7 to the embedded vertices of $\gamma'$ to get an embedding of the edges of $\gamma'$ such that the resulting embedding of $\gamma'$ is setwise invariant under $G$.  Now by omitting the vertices of $\gamma'-\gamma$ we obtain the required embedding of $\gamma$.  \end{proof}

\bigskip


\section{Embeddings with $\TSG(\Gamma)\cong S_4$}

 Recall from Theorem \ref{S4necessary} that if $K_m$ has an embedding $\Gamma$ with $\TSG(\Gamma)\cong S_4$, then $m\equiv 0$, $4$, $8$, $12$, $20\pmod{24}$.  For each of these values of $m$, we will use the Edge Embedding Lemma to construct an embedding of $K_m$ whose topological symmetry group is isomorphic to $S_4$.

\bigskip

\begin{prop}\label{T:K12} Let $m\equiv 0$, $4$, $8$, $12$, $20\pmod{24}$. Then there is an embedding $\Gamma$ of $K_m$ in $S^3$ such that $\TSG(\Gamma)\cong S_4$.  \end{prop}

\begin{proof}  Let $G\cong S_4$ be the finite group of orientation preserving isometries of $S^3$ which leaves the 1-skeleton $\tau$ of a tetrahedron setwise invariant.    Observe that every non-trivial element of $G$ with non-empty fixed point set is conjugate to one of the rotations $f$, $h$, $g$, or $g^2$ illustrated in Figure \ref{K12}.   Furthermore, an even order element of $G$ has non-empty fixed point set if and only if it is an involution.  Also, for every $g \in G$ of order 4, every point in $S^3$ has an orbit of size 4 under $g$, so $g$ does not interchange any pair of vertices.Thus regardless of how we embed our vertices, hypothesis (5) of the Edge Embedding Lemma will be satisfied.

\begin{figure} [h]
\includegraphics{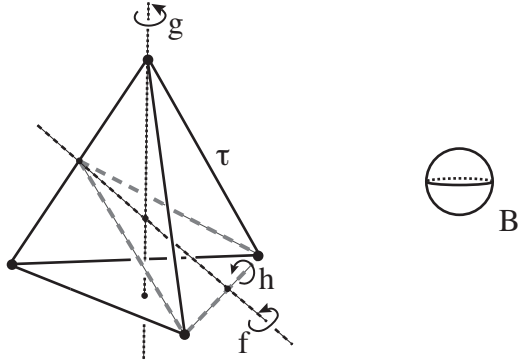}
\caption{$G$ leaves the 1-skeleton $\tau$ of a tetrahedron setwise invariant.  The ball $B$ is disjoint from the fixed point set of any non-trivial element of $G$ and from its image under every non-trivial element of $G$.}
\label{K12}
\end{figure}

Let $n\geq 0$.  We begin by defining an embedding of $K_{24n}$.   Let $B$ denote a ball which is disjoint from the fixed point set of any non-trivial element of $G$, and which is disjoint from its image under every non-trivial element of $G$.  Choose $n$ points in $B$, and let $V_0$ denote the orbit of these points under $G$.  Since $|S_4|=24$, the set $V_0$ contains $24n$ points.  These points will be the embedded vertices of $K_{24n}$.  Since none of the points in $V_0$ is fixed by any non-trivial element of $G$, it is easy to check that hypotheses (1) - (4) of the Edge Embedding Lemma are satisfied for the set $V_0$.  Thus the Edge Embedding Lemma gives us an embedding $\Gamma_0$ of $K_{24n}$ which is setwise invariant under $G$.  It follows that $\TSG(\Gamma_0)$ contains a subgroup isomorphic to $S_4$.  However, we know by the Complete Graph Theorem that $S_4$ cannot be isomorphic to a proper subgroup of $\TSG(\Gamma_0)$.  Thus $S_4\cong \TSG(\Gamma_0)$.

Next we will embed $K_{24n+4}$.  Let $V_4$ denote the four corners of the tetrahedron $\tau$ (illustrated in Figure \ref{K12}).  We embed the vertices of $K_{24n+4}$ as the points in $V_4\cup V_0$.  Now the edges of $\tau$ are the arcs required by hypothesis (2) of the Edge Embedding Lemma.  Thus it is not hard to check that the set $V_4\cup V_0$ satisfies the hypotheses of the Edge Embedding Lemma.  By applying the Edge Embedding Lemma and the Complete Graph Theorem as above we obtain an embedding $\Gamma_4$ of $K_{24n+4}$ such that $S_4\cong \TSG(\Gamma_4)$.

\begin{figure} [h]
\includegraphics{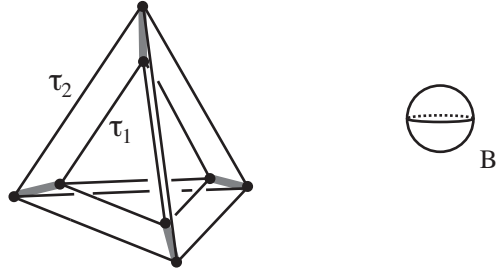}
\caption{The points of $V_8$ are the vertices of $\tau_1\cup\tau_2$.  The arcs required by hypothesis (2) are the gray arcs between corresponding vertices.  }
\label{K8S4}
\end{figure}

Next we will embed $K_{24n+8}$.  Let $T$ denote a regular solid tetrahedron with 1-skeleton $\tau$.   Let $\tau_1$ denote the 1-skeleton of a tetrahedron contained in $T$ and let $\tau_2$ denote the 1-skeleton of a tetrahedron in $S^3-T$ such that $\tau_1\cup \tau_2$ is setwise invariant under $G$.  Observe that $\tau_1$ and $\tau_2$ are interchanged by all elements conjugate to $h$ in Figure 1, and each $\tau_i$ is setwise fixed by all the other elements of $G$.  We obtain the graph illustrated in Figure \ref{K8S4} by connecting $\tau_1$ and $\tau_2$ with arcs contained in the fixed point sets of the elements of $G$ of order 3.  Now let $V_8$ denote the vertices of $\tau_1\cup \tau_2$.  Then $V_8$ is setwise invariant under $G$.  We embed the vertices of $K_{24n+8}$ as the points of $V_8\cup V_0$.  It is easy to check that hypothesis (1) of the Edge Embedding Lemma is satisfied.  To check hypothesis (2), first observe that the only pairs of vertices that are fixed by a non-trivial element of $G$ are the pairs of endpoints of the arcs joining $\tau_1$ and $\tau_2$ (illustrated as gray arcs in Figure \ref{K8S4}).  These arcs are precisely those required by hypothesis (2).  Now hypotheses (3) and (4) follow easily.  Thus again by applying the Edge Embedding Lemma and the Complete Graph Theorem we obtain an embedding $\Gamma_8$ of $K_{24n+8}$ such that $S_4\cong \TSG(\Gamma_8)$.

\begin{figure} [h]
\includegraphics{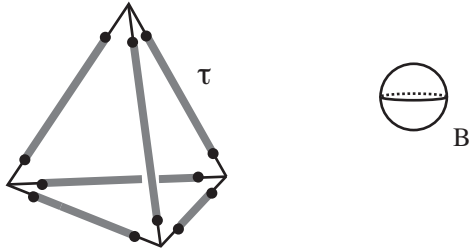}
\caption{The points of $V_{12}$ are symmetrically placed on the edges of $\tau$.  The arcs required by hypothesis (2) are the gray arcs between vertices on the same edge of $\tau$. }
\label{K12S4}
\end{figure}

Next we will embed $K_{24n+12}$.  Let $V_{12}$ be a set of 12 vertices which are symmetrically placed on the edges of the tetrahedron $\tau$ so that $V_{12}$ is setwise invariant under $G$ (see Figure \ref{K12S4}). We embed the vertices of $K_{24n+12}$ as $V_{12}\cup V_0$.  It is again easy to check that hypothesis (1) of the Edge Embedding Lemma is satisfied by $V_{12}\cup V_0$.  To check hypothesis (2) observe that the only pairs of vertices that are fixed by a non-trivial element of $G$ are pairs on the same edge of $\tau$.  The arcs required by hypothesis (2) in Figure \ref{K12S4} are illustrated as gray arcs.  Hypotheses (3) and (4) now follow easily.  Thus again by applying the Edge Embedding Lemma and the Complete Graph Theorem we obtain an embedding $\Gamma_{12}$ of $K_{24n+12}$ such that $S_4\cong \TSG(\Gamma_{12})$.

 \begin{figure} [h]
\includegraphics{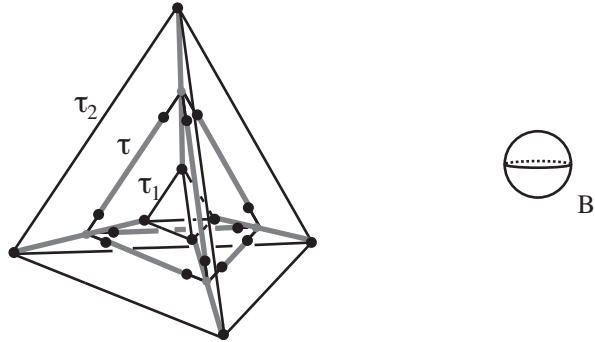}
\caption{The points of $V_8\cup V_{12}$ are the vertices of $\tau\cup\tau_1\cup\tau_2$.  The gray arcs required by hypothesis (2) are the union of those in Figures 2 and 3. }
\label{K20S4}
\end{figure}

 Finally, in order to embed $K_{24n+20}$ we first embed the vertices as $V_0\cup V_8\cup V_{12}$ from Figures \ref{K8S4} and \ref{K12S4}.  In Figure \ref{K20S4}, the 20 vertices of $V_8\cup V_{12}$ are indicated by black dots and the arcs required by hypothesis (2) are highlighted in gray.  These vertices and arcs are the union of those illustrated in Figures \ref{K8S4} and \ref{K12S4}. Now again by applying the Edge Embedding Lemma and the Complete Graph Theorem we obtain an embedding $\Gamma_{12}$ of $K_{24n+8}$ such that $S_4\cong \TSG(\Gamma_{12})$.  \end{proof}

  \bigskip

  The following theorem summarizes our results on when a complete graph can have an embedding whose topological symmetry group is isomorphic to $S_4$.

\begin{S4}
A complete graph $K_m$ with $m\geq 4$ has an embedding $\Gamma$ in $S^3$ such that  $\TSG(\Gamma) \cong  S_4$ if and only if $m \equiv 0$, $4$, $8$, $12$, $20 \pmod {24}$.
\end{S4}
\bigskip

\section{Embeddings $\Gamma$ with $\TSG(\Gamma)\cong A_5$} \label{S:A5}

Recall from Theorem \ref{A5necessary} that if $K_m$ has an embedding $\Gamma$ in $S^3$ such that $G=\TSG(\Gamma)\cong A_5$, then $m\equiv 0$, 1,  5, or  20 $\pmod {60}$.  In this section we show that for all of these values of $m$ there is an embedding of $K_m$ whose topological symmetry group is isomorphic to  $A_5$.

\begin{prop}\label{P:A5}
Let $m\equiv 0$, $1$, $5$, $20 \pmod{60}$.  Then there exists an embedding $\Gamma$ of $K_m$ in $S^3$ such that $\TSG(\Gamma)\cong A_5$.
\end{prop}

\begin{proof}   Let $G\cong A_5$ denote the finite group of orientation preserving isometries of $S^3$ which leaves a regular solid dodecahedron $D$ setwise invariant.  Every element of this group is a rotation, and hence has non-empty fixed point set.  Also the only even order elements of $A_5$ are involutions.  Thus regardless of how we embed our vertices, hypothesis (5) of the Edge Embedding Lemma will be satisfied for the group $G$.  Let $H\cong A_5$ denote the finite group of orientation preserving isometries of $S^3$ which leaves a regular $4$-simplex $\sigma$ setwise invariant.   Observe that the elements of order 2 of $H$ interchange pairs of vertices of the 4-simplex and hence have non-empty fixed point sets.  Thus regardless of how we embed our vertices, hypothesis (5) of the Edge Embedding Lemma will be satisfied for the group $H$.  We will use either $G$ or $H$ for each of our embeddings.

We shall use $G$ to embed $K_{60n}$.  Let $B$ be a ball which is disjoint from the fixed point set of any non-trivial element of $G$ and which is disjoint from its image under every non-trivial  element of $G$.  Choose $n$ points in $B$, and let $V_0$ denote the orbit of these points under $G$.  We embed the vertices of $K_{60n}$ as the points of $V_0$.  Since none of the points of $V_0$ is fixed by any non-trivial element of $G$, the hypotheses of the Edge Embedding Lemma are easy to check.  Thus by applying the Edge Embedding Lemma and the Complete Graph Theorem, we obtain an embedding $\Gamma_0$ of $K_{60n}$ in $S^3$ such that  $A_5\cong \TSG(\Gamma_0)$.

In order to embed $K_{60n+1}$ we again use the isometry group $G$.  We embed the vertices of $K_{60n+1}$ as $V_0\cup \{x\}$, where $x$ is the center of the invariant solid dodecahedron $D$.  Since $x$ is the only vertex which is fixed by a non-trivial element of $G$, the hypotheses of the Edge Embedding Lemma are satisfied for $V_0\cup \{x\}$.  Thus as above, by applying the Edge Embedding Lemma and the Complete Graph Theorem, we obtain an embedding $\Gamma_1$ of $K_{60n+1}$ in $S^3$ such that $A_5\cong \TSG(\Gamma_1)$.

In order to embed $K_{60n+5}$ we use the isometry group $H$.  Let $B'$ be a ball which is disjoint from the fixed point set of any non-trivial element of $H$ and which is disjoint from its image under every non-trivial  element of $H$.  Thus $B'$ is disjoint from the 4-simplex $\sigma$.  Choose $n$ points in $B'$, and let $W_0$ denote the orbit of these points under $H$.   Let $W_5$ denote the set of vertices of the 4-simplex $\sigma$.   We embed the vertices of $K_{60n+5}$ as the points of $W_0\cup W_5$.  Now $W_0\cup W_5$ is setwise invariant under $H$, and $H$ induces a faithful action of $K_{60n+5}$.  The arcs required by hypothesis (2) of the Edge Embedding Lemma are the edges of the  4-simplex $\sigma$.  Thus it is easy to check that the hypotheses of the Edge Embedding Lemma are satisfied for $W_0\cup W_5$.  Hence as above by applying the Edge Embedding Lemma and the Complete Graph Theorem, we obtain an embedding $\Gamma_5$ of $K_{60n+5}$ in $S^3$ such that $A_5\cong \TSG(\Gamma_5)$. 
\begin{figure} [h]
\includegraphics{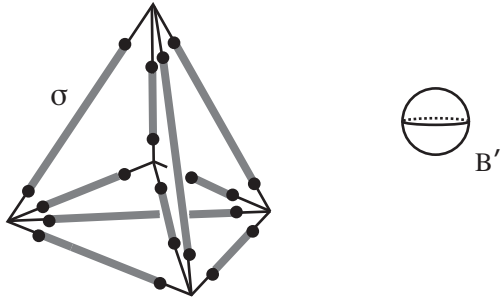}
\caption{The points of $W_{20}$ are symmetrically placed on the 4-simplex $\sigma$.  The arcs required by hypothesis (3) are the gray arcs between vertices on the same edge of $\sigma$.  }
\label{K20A5}
\end{figure}

Finally, in  order to embed $K_{60n+20}$ we again use the isometry group $H$.  Observe that each order 2 element of $H$ fixes one vertex of $\sigma$, each order 3 element of $H$ fixes 2 vertices of $\sigma$, and each order 5 element of $H$ fixes no vertices of $\sigma$.    Let $W_{20}$ denote a set of 20 points which are symmetrically placed on the edges of the 4-simplex $\sigma$ so that $W_{20}$ is setwise invariant under $H$ (see Figure 5).  We embed the vertices of $K_{60n+20}$ as the points of $W_0\cup W_{20}$.   The only pairs of points in $W_{20}$ that are both fixed by a single non-trivial element of $H$ are on the same edge of the 4-simplex $\sigma$ and are fixed by two elements of order 3.  We illustrate the arcs required by hypothesis (2) of the Edge Embedding Lemma as gray arcs in Figure \ref{K20A5}.  Now as above, by applying the Edge Embedding Lemma and the Complete Graph Theorem, we obtain an embedding $\Gamma_{20}$ of $K_{60n+20}$ in $S^3$ such that $A_5\cong \TSG(\Gamma_{20})$.\end{proof}

\bigskip

The following theorem summarizes our results on when a complete graph can have an embedding whose topological symmetry group is isomorphic to $A_5$.

\begin{A5}
A complete graph $K_m$ with $m\geq 4$ has an embedding $\Gamma$ in $S^3$ such that $\TSG(\Gamma) \cong A_5$ if and only if $m \equiv 0$, $1$, $5$, $20 \pmod{60}$.
\end{A5}
\bigskip


\section{Embeddings $\Gamma$ with $\TSG(\Gamma) \cong A_4$}

Recall from Theorem \ref{A4necessary} that if $K_m$ has an embedding $\Gamma$ in $S^3$ such that $G=\TSG(\Gamma)\cong A_4$, then $m\equiv 0$, 1, 4, 5, or $8 \pmod {12}$.   In this section we will show that for these values of $m$, there are embeddings of $K_m$ whose topological symmetry group is isomorphic to  $A_4$. 

We begin with the special case of $K_4$.  First, we embed the vertices of $\Gamma$ at the corners of a regular tetrahedron.  Then, we embed the edges of $\Gamma$ so that each set of three edges whose vertices are the corners of a single face of the tetrahedron are now tangled as shown in Figure~\ref{F:face}.  The two dangling ends on each side in Figure~\ref{F:face} continue into an adjacent face with the same pattern.  If we consider the knot formed by the three edges whose vertices are the corners of a single face (and ignore the other edges of $\Gamma$), we see that this cycle is embedded as the knot $K$ illustrated in Figure \ref{F:knot}.

\begin{figure} [h]
\includegraphics{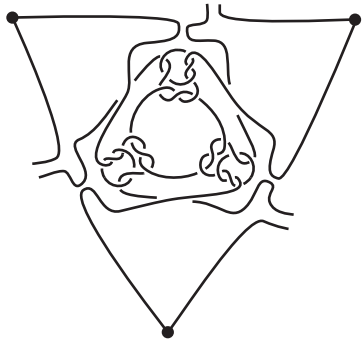}
\caption{One face of the embedding $\Gamma$ of $K_4$.} \label{F:face}
\end{figure}

\begin{figure} [h]
\includegraphics{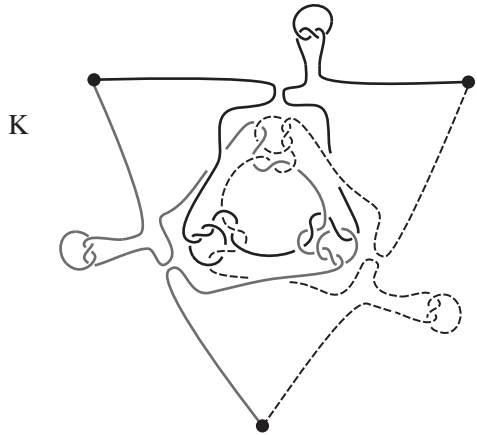}
\caption{The knot $K$ formed by three edges whose vertices are the corners of a single face.   We indicate the three edges of the triangle with different types of lines.} \label{F:knot}
\end{figure}

\begin{lemma}\label{noninvertible}
The knot $K$ in Figure~\ref{F:knot} is non-invertible.
\end{lemma}

\begin{proof}
Observe that $K$ is the connected sum of three trefoil knots together with the knot $J$ illustrated in Figure \ref{prime}.
Suppose $K$ is invertible. Then by the uniqueness of prime factorizations of oriented knots, $J$ would also be invertible.  Since $J$ is the closure of the sum of three rational tangles, $J$ is an algebraic knot.  Thus the machinery of Bonahon and Siebenmann \cite{BS} can be used to show that $J$ is non-invertible.  It follows that $K$ is non-invertible as well.  
\begin{figure} [h]
\includegraphics{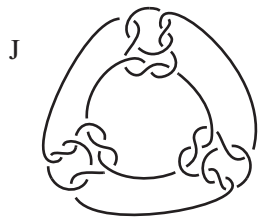}
\caption{$K$ is the connected sum of three trefoil knots together with the knot $J$ illustrated here.} \label{prime}
\end{figure}
\end{proof}
\bigskip

\begin{prop} \label{T:K4}
Let $\Gamma$ be the embedding of $K_4$ in $S^3$ described above.  Then $\TSG(\Gamma) \cong  A_4$, and $\TSG(\Gamma)$ is induced by the group of rotations of a solid tetrahedron.\end{prop}

\begin{proof}  It follows from the Complete Graph Theorem that if $A_4$ is isomorphic to a subgroup of $ \TSG(\Gamma)$, then $\TSG(\Gamma)$ is isomorphic to either $A_4$, $S_4$, or $A_5$.
We will first show that $ \TSG(\Gamma)$ contains a group isomorphic to $A_4$, and then that $\TSG(\Gamma)$ is not isomorphic to either $S_4$ or $A_5$.

We see as follows that $\Gamma$ is setwise invariant under a group of rotations of a solid tetrahedron.   The fixed point set of an order three rotation of a solid tetrahedron contains a single vertex of the tetrahedron and a point in the center of the face opposite that vertex.  To see that $\Gamma$ is invariant under such a rotation, we unfold three of the faces of the tetrahedron.  The unfolded picture of $\Gamma$ is illustrated in Figure \ref{F:K4net1}.  In order to recover the embedded graph $\Gamma$ from Figure \ref{F:K4net1}, we glue together the pairs of sides with corresponding labels.  When we re-glue these pairs, the three vertices labeled $x$ become a single vertex.  We can see from the unfolded picture in Figure \ref{F:K4net1} that there is a rotation of $\Gamma$ of order three which fixes the point $y$ in the center of the picture, together with the vertex $x$.  
\begin{figure} [h]
\includegraphics{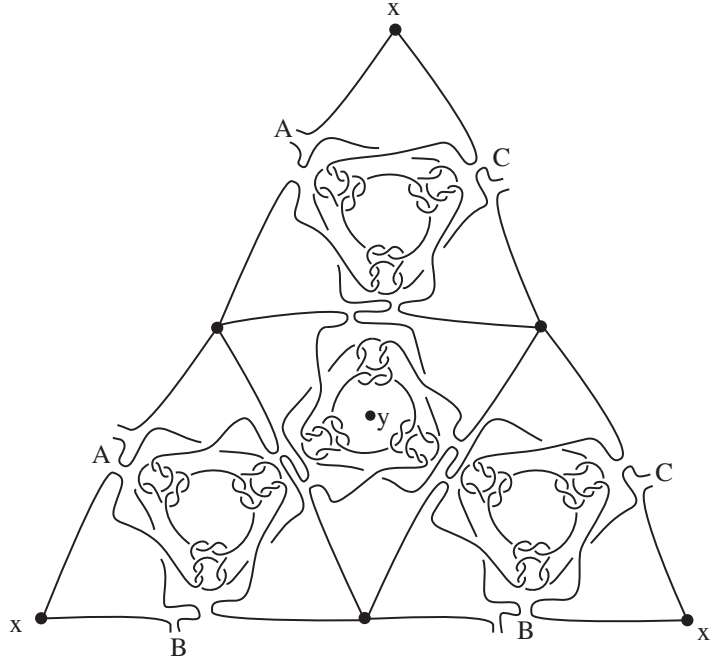}
\caption{This unfolded view illustrates an order three symmetry of $\Gamma$.} \label{F:K4net1}
\end{figure}

The fixed point set of a rotation of order two of a tetrahedron contains the midpoints of two opposite edges.  This rotation interchanges the two faces which are adjacent to each of these inverted edges.  To see that $\Gamma$ is invariant under such a rotation, we unfold the tetrahedron into a strip made up of four faces of the tetrahedron.  The unfolded picture of $\Gamma$ is illustrated in Figure \ref{F:K4net2}.  In order to recover the embedded graph $\Gamma$ from Figure \ref{F:K4net2}, we glue together pairs of sides with corresponding labels.  When we glue these pairs, the two points labeled $v$ are glued together.  We can see from the unfolded picture in Figure \ref{F:K4net2} that there is a rotation of $\Gamma$ of order two which fixes the point $w$ that is in the center of the picture together with the point $v$.  Thus $\TSG(\Gamma)$ contains a subgroup isomorphic to $A_4$.

\begin{figure} [h]
\includegraphics{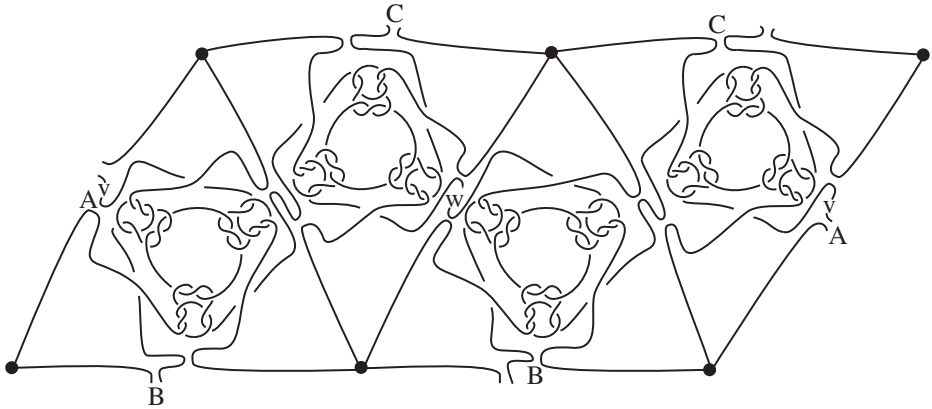}
\caption{This unfolded view illustrates an order two symmetry of $\Gamma$.} \label{F:K4net2}
\end{figure}

Now assume that $S_4\cong\TSG(\Gamma)$.  Label the four vertices of $\Gamma$ by the letters $a$, $b$, $c$, and $d$.  Then there is a homeomorphism $h$ of $S^3$ which leaves $\Gamma$ setwise invariant while inducing the automorphism $(ab)$ on its vertices.  In particular, the image of the oriented cycle $abc$ is the oriented cycle $bac$.  Thus the simple closed curve in $\Gamma$ with vertices $abc$ is inverted by $h$.   However, this simple closed curve is the knot $K$ illustrated in Figure~\ref{F:knot}, and we proved in Lemma \ref{noninvertible} that $K$ is non-invertible.  Therefore, $S_4 \not\cong TSG(\Gamma)$.

Finally, by Theorem 2, $A_5 \not \cong  \TSG(\Gamma)$, which completes the proof.
\end{proof}
\bigskip

Now we will show that for all $m>4$ such that $m\equiv 0,1,4, 5, 8\pmod {12}$, there is an embedding $\Gamma$ of $K_m$ in $S^3$ with $\TSG(\Gamma)\cong A_4$. The following Theorem from \cite{FMN} will be used in the proof.

\bigskip

\begin{Subgroup}\cite{FMN} \label{L:fundamentaledge}
Let $\Gamma$ be an embedding of a 3-connected graph in $S^3$.  Suppose that $\Gamma$ contains an edge $e$ which is not pointwise fixed by any non-trivial element of $\TSG(\Gamma)$.  Then for every $H\leq \TSG(\Gamma)$, there is an embedding $\Gamma'$ of $\Gamma$ with $H = \TSG(\Gamma')$.
\end{Subgroup}

\bigskip

\begin{prop}\label{T:A4} Suppose that $m>4$ and $m\equiv 0,1,4, 5, 8\pmod {12}$.  Then there is an embedding $\Gamma$ of $K_m$ in $S^3$ such that $\TSG(\Gamma)\cong A_4$. \end{prop}

\begin{proof} We first consider the cases where $m\equiv 0, 4,8,12, 20 \pmod{24}$.  Let $G$ denote the finite group of orientation preserving isometries of the 1-skeleton $\tau$ of a regular tetrahedron. Recall from the proof of Proposition \ref{T:K12} that for each $k=$ 0, 4, 8, 12, or 20, we embedded $K_{24n+k}$ as a graph $\Gamma_k$ with vertices in the set $V_0\cup V_4\cup V_8\cup V_{12}$ such that $\TSG(\Gamma_k)\cong S_4$ is induced by $G$.  We will show that each $\Gamma_k$ has an edge which is is not pointwise fixed by any non-trivial element of $G$.   

First suppose that $m=24n+k$ where $n>0$ and $k=$ 0, 4, 8, 12, or 20.   Recall that $V_0$ contains $24n$ vertices none of which is fixed by any non-trivial element of $G$.  Let $e_0$ be an edge of $\Gamma_k$ with vertices in $V_0$.  Then $e_0$ is not pointwise fixed by any non-trivial element of $G$.  Hence by the Subgroup Theorem, there is an embedding $\Gamma$ of $K_{m}$ with $A_4 \cong  \TSG(\Gamma)$.

  Next we suppose that $n=0$, and let $\Gamma_8$, $\Gamma_{12}$, and $\Gamma_{20}$ denote the embeddings of $K_8$, $K_{12}$, and $K_{20}$ given in the proof of Proposition \ref{T:K12}.  Let $e_8$ be an edge in $\Gamma_8$ whose vertices are not the endpoints of one of the gray arcs in Figure \ref{K8S4}, let $e_{12}$ be an edge in $\Gamma_{12}$ whose vertices are not the endpoints of one of the gray arcs in Figure \ref{K12S4}, and let $e_{20}$ be an edge in $\Gamma_{20}$ whose vertices are not the endpoints of one of the gray arcs in Figure \ref{K20S4}. In each case, $e_k$ is not pointwise fixed by any non-trivial element of $G$.  Hence by the Subgroup Theorem, there is an embedding $\Gamma$ of $K_{k}$ with $A_4 \cong  \TSG(\Gamma)$.

Next we consider the case where $m=24n+16$.  Then $m=12(2n+1)+4$.  Let $H\cong A_4$ denote a finite group of orientation preserving isometries of $S^3$ which leaves a solid tetrahedron $T$ setwise invariant.   Then every element of $H$ has non-empty fixed point set, and the only even order elements are involutions.  Thus regardless of how we embed our vertices, hypothesis (5) of the Edge Embedding Lemma will be satisfied for $H$.  Also, observe that no edge of the tetrahedron $T$ is pointwise fixed by any non-trivial element of $H$. Let $W_4$ denote the vertices of $T$.   Let $B$ denote a ball which is disjoint from the fixed point set of any non-trivial element of $H$, and which is disjoint from its image under every non-trivial element of $H$. Choose $2n+1$ points in $B$ (recall that we are not assuming that $n>0$) and let $W_0$ denote the orbit of these points under $H$.   We embed the vertices of $K_{12(2n+1)+4}$ as the points of $W_0\cup W_4$.  Since no pair of vertices in $W_0\cup W_4$ are both fixed by a non-trivial element $h\in H$, it is easy to see that  hypotheses (1) - (4) of the Edge Embedding Lemma are satisfied.  Thus by applying the Edge Embedding Lemma we obtain an embedding $\Gamma_{16}$ of $K_{24n+16}$ which is setwise invariant under $H$.  Now by Theorems 2 and 3, we know that $\TSG(\Gamma_{16})\not\cong A_5$ or $S_4$.  Now it follows from the Complete Graph Theorem that $\TSG(\Gamma_{16})\cong A_4$.  

Thus we have shown that if $m\equiv 0$, 4, or 8 $\pmod {12}$ and $m>4$, then there is an embedding $\Gamma$ of $K_m$ in $S^3$ with $\TSG(\Gamma)\cong A_4$.  

Next suppose that $m=12n+1$ and $m\not \equiv 1\pmod{60}$.  Let the group $H$ and the ball $B$ be as in the above paragraph.  Choose $n$ points in $B$ and let $U_0$ denote the orbit of these points under $H$.  Let $v$ denote one of the two points of $S^3$ which is fixed by every element of $H$.  We embed the vertices of $K_m$ as $U_0\cup \{v\}$.  Since no pair of vertices in $U_0\cup \{v\}$ are both fixed by a non-trivial element $h\in H$, it is easy to check that the hypotheses of the Edge Embedding Lemma are satisfied.  Now by the Edge Embedding Lemma together with Theorems 2 and 3 and the Complete Graph Theorem we obtain an embedding $\Gamma$ of $K_m$ such that $\TSG(\Gamma)\cong A_4$.  

Similarly, suppose that $m=12n+5$ and $m\not \equiv 5\pmod{60}$.  Let $H$, $U_0$, $W_4$, and $v$ be as above.  We embed the vertices of $K_m$ as $U_0\cup W_4\cup \{v\}$.  The arcs required by hypothesis (2) of the Edge Embedding Lemma are highlighted in gray in Figure \ref{K5A4}.  Now again by the Edge Embedding Lemma together with Theorems 2 and 3 and the Complete Graph Theorem we obtain an embedding $\Gamma$ of $K_m$ such that $\TSG(\Gamma)\cong A_4$.

\begin{figure} [h]
\includegraphics{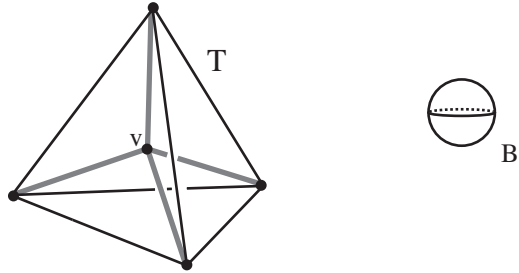}
\caption{The vertices of $W_4\cup \{v\}$ are indicated by black dots.  The arcs required by hypothesis (3) are highlighted in gray. }
\label{K5A4}
\end{figure}

Next suppose that $m=60n+1$ or $m=60n+5$ where $n>0$.  Let $G_1$ denote the group of orientation preserving symmetries of a regular solid dodecahedron and let $G_5$ denote the  group of orientation preserving symmetries of a regular 4-simplex.  We first embed $K_{60n+1}$ and $K_{60n+5}$ as the graphs $\Gamma_1$ and $\Gamma_5$ respectively given in the proof of Proposition \ref{P:A5} such that $\TSG(\Gamma_k)\cong A_5$ is induced by $G_k$, where $k = 1, 5$.  Since $n>0$, we can choose an edge $e_0$ of $\Gamma_k$ both of whose vertices are in $V_0$.  Then $e_0$ is not pointwise fixed by any non-trivial element of $G_k$ .  Hence by the Subgroup  Theorem, we obtain an embedding $\Gamma$ of $K_{m}$ such that $\TSG(\Gamma)\cong A_4$.

Finally, let $m=5$.  Let $\mu$ denote an embedding of the 1-skeleton of a regular solid tetrahedron $T$ so that the edges of $\mu$ each contain an identical trefoil knot.  Let $\Gamma$ denote these vertices and edges together with a vertex at the center of $T$ which is connected via unknotted arcs to the other vertices of $T$ (see Figure \ref{K5}).  We choose $\Gamma$ so that it is setwise invariant under a group of orientation preserving isometries of $T$.  Thus $\TSG(\Gamma)$ contains a subgroup isomorphic to $A_4$.  Since $\Gamma$ is an embedding of $K_5$, by Theorem 2  we know that $\TSG(\Gamma)\not \cong S_4$.  Furthermore, any homeomorphism of $(S^3,\Gamma)$, must take each triangle which is the connected sum of 3 trefoil knots to a triangle which also is the connected sum of 3 trefoil knots.  Thus $\TSG(\Gamma)$ must leave $\mu$ setwise invariant.  Since $\mu$ is an embedding of $K_4$, $\TSG(\Gamma)$ induces a faithful action on $K_4$.  However, $A_5$ cannot act faithfully on $K_4$.  Thus $\TSG(\Gamma)\not\cong A_5$.  Now it follows from the Completeness Theorem that  $\TSG(\Gamma)\cong A_4$.  

 \begin{figure} [h]
\includegraphics{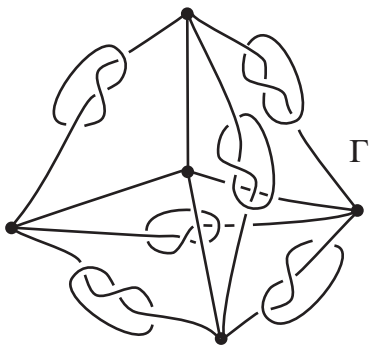}
\caption{An embedding $\Gamma$ of $K_5$ such that $\TSG(\Gamma)=A_4$. }
\label{K5}
\end{figure}

The above four paragraphs together show that if $m\equiv 1, 5\pmod {12}$ and $m>4$, then there is an embedding $\Gamma$ of $K_m$ in $S^3$ with $\TSG(\Gamma)\cong A_4$.  
\end{proof}

\bigskip

The following theorem summarizes our results on when a complete graph can have an embedding whose topological symmetry group is isomorphic to $A_4$.

\begin{A4}
A complete graph $K_m$ with $m\geq 4$ has an embedding $\Gamma$ in $S^3$ such that  $\TSG(\Gamma) \cong  A_4$ if and only if $m \equiv 0$, $1$, $4$, $5$, $8 \pmod {12}$.\end{A4}

\end{document}